\documentclass[12pt,a5paper]{article}

\usepackage{amsmath,amsfonts,amsthm,defns,ajr}
\allowdisplaybreaks

\newcommand{\DW}{{\Delta W}}
\newcommand{\intw}[2]{\int_0^{\ifx#1t h\else%
\ifx#1s t\else\ifx#1r s\else\ifx#1q r\else 
\fi \fi \fi \fi}\!#2\,dW_{#1}}
\newcommand{\intab}{\int_a^b}
\def\E{\operatorname{E}}

\def\Var{\operatorname{Var}}

\theoremstyle{definition}
\newtheorem{exercise}{Exercise}
\topsep=\parskip

\title{Modify the Improved Euler scheme to integrate  stochastic differential equations}

\author{A.~J. Roberts\thanks{School of Mathematical Sciences, University of Adelaide, South Australia. \url{maitlo:anthony.roberts@adelaide.edu.au}}}

\date{October 3, 2012}

\begin{document}

\maketitle

\begin{abstract}
A practical and new Runge--Kutta numerical scheme for stochastic differential equations is explored.  Numerical examples demonstrate the strong convergence of the method.  The first order strong convergence is then proved using It\^o integrals for both It\^o and Stratonovich interpretations.  As a straightforward modification of the deterministic Improved Euler/Heun method, the method is a good entry level scheme for stochastic differential equations, especially in conjunction with Higham's introduction [SIAM Review, 43:525--546, 2001].
\end{abstract}

\tableofcontents

\section{Introduce a modified integration}

Nearly twenty years ago Kloeden and Platen~\cite{Kloeden92} described schemes for numerically integrating stochastic differential equations (\sde{}s).  Intervening research led to recent developments of useful Runge--Kutta like methods for It\^o \sde{}s by Andreas Rossler~\cite{Rossler2010, Rossler2009}  and for Stratonovich \sde{}s by Yoshio Komori~\cite{Komori2007c, Komori2007b, Komori2007a}.
These numerical integration schemes for \sde{}s are quite complicated, and typically do not easily reduce to accurate deterministic schemes.
This short article introduces a Runge--Kutta scheme for \sde{}s that does straightforwardly reduce to a well known deterministic scheme---the variously called Improved Euler, Heun, or Runge--Kutta~2 scheme.

As well as being a novel practical scheme for the numerical integration of \sde{}s, because of the strong connection to a well known deterministic integration scheme, the scheme proposed here serves as an entry level scheme for teaching stochastic dynamics.  One could use this scheme together with Higham's~\cite{Higham01} introduction to the numerical simulation of \sde{}s.  Section~\ref{sec:edet} on the method applied to examples assumes a background knowledge of basic numerical methods for ordinary differential equations and deterministic calculus as typically taught in early years at university.  Section~\ref{sec:peg} on the underlying theory assumes knowledge of stochastic processes such as continuous time Markov Chains, and, although not essential, preferably at least a formal introduction to stochastic differential equations (such as the book~\cite{Roberts08g} or article~\cite{Higham01} with material that is successfully taught at second\slash third year university).

Consider the vector stochastic process~$\vec X(t)\in \mathbb R^n$ that satisfies the general It\^o \sde
\begin{equation}
d\vec X=\vec a(t,\vec X)\,dt+\vec b(t,\vec X)\,dW,
\label{eq:sde1ab}
\end{equation}
where drift~$\vec a$ and volatility~$\vec b$ are sufficiently smooth functions of their arguments.  The noise is represented by the differential~$dW$ which symbolically denotes infinitesimal increments of the random walk of a Wiener process~$W(t,\omega)$.  The symbolic form of the \sde~\eqref{eq:sde1ab} follows from the most basic approximation to an evolving system with noise that over a time step~$\Delta t_k$ the change in the dependent variable is
\begin{equation*}
\Delta\vec X_k\approx \vec a(t_k,\vec X_k)\Delta t_k
+\vec b(t_k,\vec X_k)\DW_k
\end{equation*}
where $\DW_k=W(t_{k+1},\omega)-W(t_k,\omega)$ symbolises some `random' effect.  This basic approximation is low accuracy and needs improving for practical applications, but it does form a basis for theory, and it introduces the noise process~$W(t,\omega)$, called a Wiener process. We use~$\omega$ to denote the realisation of the noise.  Such a Wiener process is defined by $W(0,\omega)=0$ and that the increment $W(t,\omega)-W(s,\omega)$ is distributed as a zero-mean, normal variable, with variance~$t-s$\,, and independent of earlier times.  Consequently, crudely put, $dW/dt$~then is a `white noise' with a flat power spectrum.  The \sde~\eqref{eq:sde1ab} may then be interpreted as a dynamical system affected by white noise. 

The proposed modified Runge--Kutta scheme for the general \sde~\eqref{eq:sde1ab} is the following. Given time step~$h$, and given the value $\vec X(t_k)=\vec X_k$\,, estimate $\vec X(t_{k+1})$ by~$\vec X_{k+1}$ for time $t_{k+1}=t_k+h$ via
\begin{align}
&\vec K_1=h\vec a(t_k,\vec X_k)+(\DW_k-S_k\sqrt h)\vec b(t_k,\vec X_k),
\nonumber\\&\vec K_2=h\vec a(t_{k+1},\vec X_k+\vec K_1)+(\DW_k+S_k\sqrt h)\vec b(t_{k+1},\vec X_k+\vec K_1),
\nonumber\\&\vec X_{k+1}=\vec X_k+\rat12(\vec K_1+\vec K_2),
\label{eq:ieuabj}
\end{align}
\begin{itemize}
\item where $\DW_k=\sqrt hZ_k$ for normal random $Z_k\sim N(0,1)$; 
\item and where $S_k=\pm1$\,, each alternative chosen with probability~$1/2$.
\end{itemize}
The above describes only one time step.
Repeat this time step $(t_m-t_0)/h$~times in order to integrate an \sde~\eqref{eq:sde1ab} from time $t=t_0$ to $t=t_m$\,.

The appeal of the scheme~\eqref{eq:ieuabj} as an entry to stochastic integrators is its close connection to deterministic integration schemes.
When the stochastic component vanishes, $\vec b=\vec 0$\,, the integration step~\eqref{eq:ieuabj} is precisely the Improved Euler, Heun, or Runge--Kutta~2 scheme that most engineering, science and mathematics students learn in undergraduate coursework.

This connection has another useful consequence in application: for systems with small noise we expect that the integration error of the \sde\ is only a little worse than that of the deterministic system.  Although Section~\ref{sec:peg} proves the typical \Ord{h}~error of the stochastic scheme~\eqref{eq:ieuabj}, as demonstrated in the examples of the next Section~\ref{sec:edet}, when the noise is small expect the error to be practically better than the order of error suggests.

Section~\ref{sec:peg} also proves that the scheme~\eqref{eq:ieuabj} integrates Stratonovich \sde{}s to~\Ord{h} provided one sets $S_k=0$ throughout (instead of choosing~$\pm 1$).

An outstanding challenge is to generalise this method~\eqref{eq:ieuabj} to multiple noise sources.

\section{Examples demonstrate \Ord{h} error is typical}
\label{sec:edet}

This section applies the scheme~\eqref{eq:ieuabj} to three example \sde{}s for which, for comparison, we know the analytic solution from Kloeden and Platen~\cite{Kloeden92}.  Two of the examples exhibit errors~\Ord{h}, as is typical, whereas the third exhibits a error~\Ord{h^2}, which occurs for both deterministic \ode{}s and a class of \sde{}s. These errors are `pathwise' errors which means that for any one given realisation~$\omega$ of the noise process~$W(t,\omega)$ we refer to the order of error as the time step~$h\to0$ for a fixed realisation~$\omega$. 

\subsection{Autonomous example}
\label{sec:aeg}

Consider the `autonomous' \sde
\begin{equation}
dX=\left[\rat12X+\sqrt{1+X^2}\right]dt +\sqrt{1+X^2}\,dW
\qtq{with}X(0)=0\,,
\label{eq:aeg}
\end{equation}
for some Wiener process~$W(t,\omega)$. The \sde\ is not strictly autonomous because the noise~$dW$ introduces time dependence; we use the term `autonomous' to indicate the drift~$a$ and volatility~$b$ are independent of time.  For the \sde~\eqref{eq:aeg}, Kloeden and Platen~\cite{Kloeden92} list the analytic solution as $X(t,\omega)=\sinh\big[t+W(t,\omega)\big]$.

Such analytic solutions are straightforwardly checked via the basic version of It\^o's formula, 
\begin{equation}
\text{if }X=f(t,w)\text{ for }w=W(t,\omega),\text{ then }
dX=\left(\D tf+\frac12\DD wf\right)\,dt +\D wf\,dW,
\label{eq:ito}
\end{equation}
which may be understood as the usual deterministic derivative rule $dX=(\D tf)\,dt +(\D wf)\,dW$ with the extra term~$\frac12(\DD wf)\,dt$ arising from a formal multi-variable Taylor series in the infinitesimals $dt$~and~$dw$, recognising formally that $dW^2=dt$ in effect, and all remaining infinitesimal products negligible~\cite[e.g.]{Higham01, Roberts08g}.

\begin{figure}
\centering
\begin{tabular}{c@{}cc}
\rotatebox{90}{\hspace{20ex}$X(t)$} &
\includegraphics[scale=1.2]{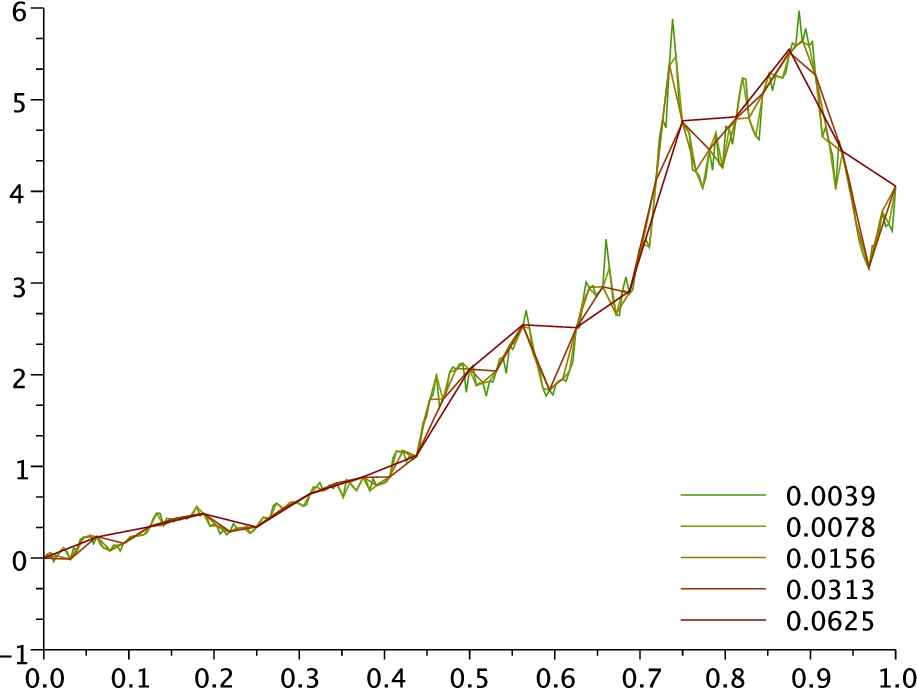}\\
& time $t$
\end{tabular}
\caption{As the time step is successively halved, $n=16,32,64,128,256$ time steps over $0\leq t\leq1$\,, the numerical solutions of the \sde~\eqref{eq:aeg} via the method~\eqref{eq:ieuabj} appear to converge.}
\label{fig:sde1absim2}
\end{figure}

The proposed numerical scheme~\eqref{eq:ieuabj} was applied to integrate the \sde~\eqref{eq:aeg} from $t=0$ to end time $t=1$ with a time step of $h=1/n$ for $n=2^{16},2^{15},\ldots,2^4$ steps.  For each of $700$~realisations of the noise~$W(t,\omega)$, the Wiener increments, $\DW\sim N(0,2^{-16})$, were generated on the finest time step, and subsequently aggregated to the corresponding increments for each realisation on the coarser time steps.  Figure~\ref{fig:sde1absim2} plots the predicted~$X(t,\omega)$ obtained from the numerical scheme~\eqref{eq:ieuabj} for just one realisation~$\omega$ using different time steps.  The predictions do appear to converge to a well defined stochastic process as the step size is repeatedly halved.

\begin{figure}
\centering
\begin{tabular}{cc}
\rotatebox{90}{\hspace{15ex}\textsc{rms} error} &
\includegraphics[scale=1.2]{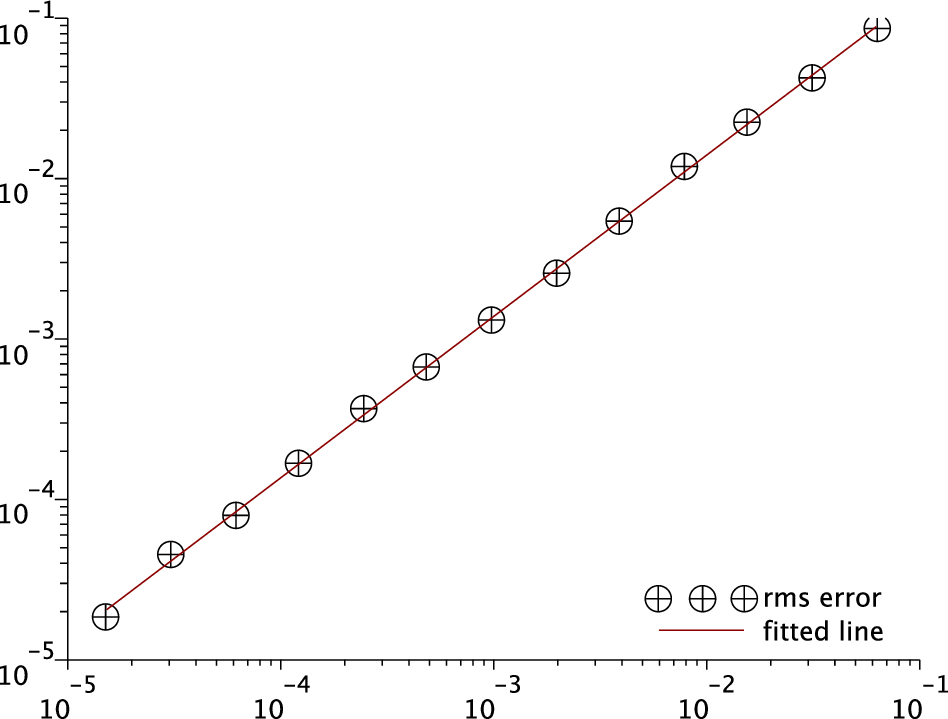}\\
& time step~$h$
\end{tabular}
\caption{Average over $700$ realisations at each of 13 different step sizes for the \sde~\eqref{eq:aeg}: at $t=1$\,, the \textsc{rms} error in the predicted~$X(1,\omega)$ decreases linearly in time step~$h$.}
\label{fig:sde1abll2}
\end{figure}

For each size of time step, Figure~\ref{fig:sde1abll2} uses the analytic solution to find the \textsc{rms} error of the predicted~$X(1,\omega)$, averaged over $700$~realisations~$\omega$.  This \textsc{rms} error estimates the square-root of the expectation $E[(X_m-X(1,\omega))^2]$. Figure~\ref{fig:sde1abll2} uses a log-log plot to show that the \textsc{rms} error decreases linearly with time step size~$h$ (over four orders of magnitude in time step).  That is, empirically we see the scheme~\eqref{eq:ieuabj} has \textsc{rms} error~\Ord{h}.

\subsection{Non-autonomous example}

Consider the `non-autonomous' \sde
\begin{equation}
dX=\left[\frac{X}{1+t}-\frac32X\left(1-\frac{X^2}{(1+t)^2}\right)^2\right]dt 
+(1+t)\left(1-\frac{X^2}{(1+t)^2}\right)^{3/2}dW,
\label{eq:naeg}
\end{equation}
with initial condition that $X(0)=0$\,,
for some Wiener process~$W(t,\omega)$.  Here both the drift~$a$ and the volatility~$b$ have explicit time dependence.  It\^o's formula~\eqref{eq:ito} confirms that the analytic solution to this \sde~\eqref{eq:aeg} is $X(t,\omega)=(1+t)W(t,\omega)/\sqrt{1+W(t,\omega)^2}$.

\begin{figure}
\centering
\begin{tabular}{cc}
\rotatebox{90}{\hspace{15ex}\textsc{rms} error} &
\includegraphics[scale=1.2]{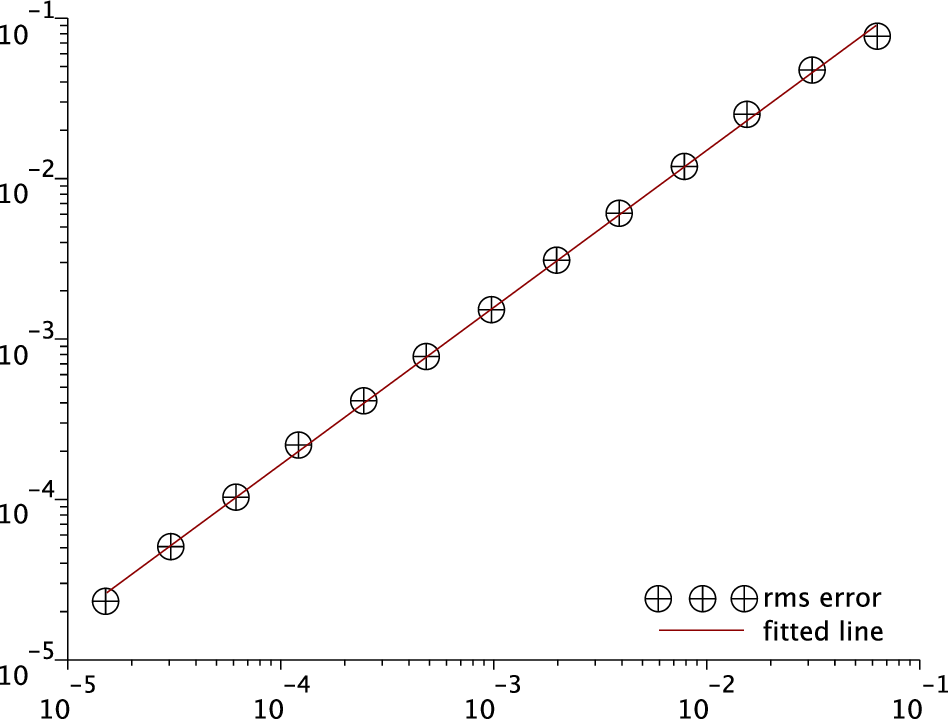}\\[-1ex]
& time step~$h$
\end{tabular}
\caption{Average over $700$ realisations at each of 13~different step sizes for the \sde~\eqref{eq:naeg}: at $t=1$\,, the \textsc{rms} error in the predicted~$X(1,\omega)$ decreases linearly in time step~$h$.}
\label{fig:sde1abll8}
\end{figure}

To determine the order of error of the scheme~\eqref{eq:ieuabj}, the same approach was adopted here as described in the previous
Section~\ref{sec:aeg}.  The slope of the log-log plot in Figure~\ref{fig:sde1abll8} shows that again the \textsc{rms} error of the predicted~$X(1,\omega)$ is~\Ord{h} for time step~$h$ over four orders of magnitude in~$h$.

\subsection{Example with second order error}
\label{sec:esoe}

Consider the following \sde\ linear in~$X$:
\begin{equation}
dX=\left[\frac{2X}{1+t}+(1+t)^2\right]dt +(1+t)^2dW
\quad\text{with }X(0)=1\,.
\label{eq:qeg}
\end{equation}
for some Wiener process~$W(t,\omega)$.  It\^o's formula~\eqref{eq:ito} confirms that the analytic solution to this \sde~\eqref{eq:qeg} is $X(t,\omega)=(1+t)^2\big[1+t+W(t,\omega)\big]$.

\begin{figure}
\centering
\begin{tabular}{cc}
\rotatebox{90}{\hspace{15ex}\textsc{rms} error} &
\includegraphics[scale=1.2]{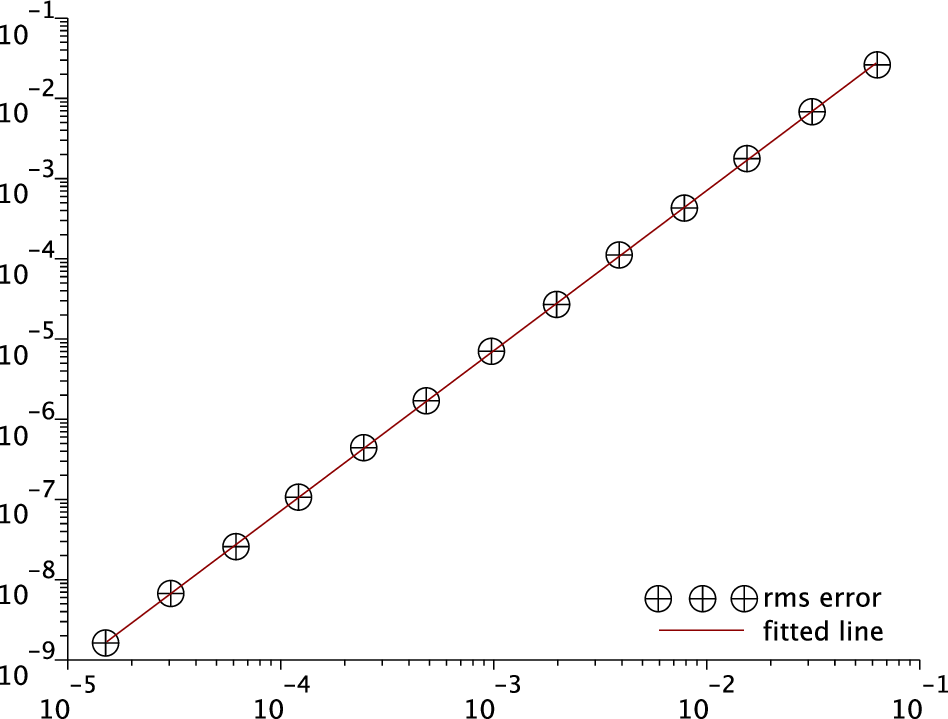}\\
& time step~$h$
\end{tabular}
\caption{Averaging over $700$ realisations at each of 13 different step sizes for the linear \sde~\eqref{eq:qeg}: at $t=1$\,, the \textsc{rms} error in the predicted~$X(1,\omega)$ decreases quadratically, like~$h^{2}$.}
\label{fig:sde1abll1}
\end{figure}

To determine the order of error of the scheme~\eqref{eq:ieuabj}, the same approach was adopted here as described in
Section~\ref{sec:aeg}.  The difference is that the slope of the log-log plot in Figure~\ref{fig:sde1abll1} shows that here the \textsc{rms} error of the predicted~$X(1,\omega)$ is~\Ord{h^2}.  There appears to be some \sde{}s for which the error of the scheme~\eqref{eq:ieuabj} is quadratic in the time step~$h$ rather than linear.

\begin{exercise} 
Use It\^o's formula~\eqref{eq:ito} to confirm the solutions given below satisfy the corresponding given \sde.
Apply the scheme~\eqref{eq:ieuabj} to some of the following \sde{}s and compare the predictions, for different time steps sizes, to the given analytic solution.  Perhaps adapt some of the code given by Higham~\cite[Listing~6]{Higham01}.
\begin{enumerate}
\item $dX=\rat12(X-t)\,dt+(X-t-2)\,dW$, $X(0)=3$; solution $X=2+t+\exp[W(t)]$.
\item $dX=X\,dW$, $X(0)=1$; solution $X=\exp[W(t)-t/2]$.
\item $dX=-X(1-X^2)\,dt+(1-X^2)\,dW$, $X(0)=0$; solution $X=\tanh[W(t)]$.
\item $dX=-X\,dt+e^{-t}dW$, $X(0)=0$; solution $X=e^{-t}W(t)$.
\item $dX=-\rat32X(1-X^2)^2dt+(1-X^2)^{3/2}dW$, $X(0)=0$; solution $X=W(t)/\sqrt{1+W(t)^2}$.
\end{enumerate} 
For which \sde{}s is the error~\Ord{h^2}?
\end{exercise}

\section{Prove \Ord{h} global error in general}
\label{sec:peg}

This section uses stochastic integration to establish the general order of accuracy of the proposed numerical integration scheme.

Proofs that numerical schemes do indeed approximate \sde\ solutions are often complex.  My plan here is to elaborate three successively more complicated cases, with the aim that you develop a feel for the analysis before it gets too complex.   Lemma~\ref{thm:rkbdw} first proves that the Runge--Kutta like scheme~\eqref{eq:ieuabj} approximates the simplest It\^o integrals $X=\intab b(t)\,dW$ to first order in the time step.  Second, section~\ref{sec:elsan} identifies a class of  linear \sde{}s with additive noise when the scheme~\eqref{eq:ieuabj} is of second order.  Third, section~\ref{sec:gegs} proves the first order global error of scheme~\eqref{eq:ieuabj} when applied to general \sde{}s.  Those familiar with stochastic It\^o integration could proceed directly to the third section~\ref{sec:gegs}.

One outcome of this section is to precisely `nail down' the requisite properties of the choice of signs~$S_j$ in the scheme~\eqref{eq:ieuabj}.

\subsection{Error for It\^o integrals}

This subsection establishes the order of error in computing the It\^o integral $X=\intab b(t)\,dW$ if one were to invoke the scheme~\eqref{eq:ieuabj} on the scalar \sde\ $dX=b(t)\,dW$.  Before proceeding, recall that two fundamental properties on the expectation and variance of It\^o integrals are widely useful~\cite[p.2]{Kloeden01} \cite[pp.101--3]{Roberts08g}:
\begin{align}
&\text{martingale property},
&&\E\left[\intab f(t,\omega)\,dW\right]=0\,;
\label{eq:marty}
\\&\text{It\^o isometry},
&&\E\left[\left(\intab f(t,\omega)\,dW\right)^2\right]
=\intab \E\left[f(t,\omega)^2\right]\,dt\,.
\label{eq:isom}
\end{align}
These empower us to quantify errors in the integrals that approximate solutions of \sde{}s as in the following lemma. 

\begin{lemma} \label{thm:rkbdw}
  The Runge--Kutta like scheme~\eqref{eq:ieuabj} has \Ord{h}~global error when applied to $dX=b(t)\,dW$ for functions~$b(t)$ twice differentiable. 
\end{lemma}

\begin{proof}
Without loss of generality, start with the time step from $t_0=0$ to $t_1=t_0+h=h$\,.  Applied to the very simple \sde\ $dX=b(t)\,dW$ one step of the scheme~\eqref{eq:ieuabj} computes
\begin{equation*}
K_1=(\DW-S\sqrt h)b_0\,,\quad
K_2=(\DW+S\sqrt h)b_1\,,
\end{equation*}
and then estimates the change in~$X$ as
\begin{equation}
\Delta\hat X=\rat12(b_0+b_1)\DW+\rat12(b_1-b_0)S\sqrt h\,,
\label{eq:dxbdw}
\end{equation}
where  the integrand values $b_0=b(0)$ and $b_1=b(h)$.
The classic polynomial approximation theorem~\cite[p.800, e.g.]{Kreyszig9} relates this estimate~\eqref{eq:dxbdw} to the exact integral.  Here write the integrand as the linear interpolant with remainder:
\begin{equation*}
b(t)=\rat12(b_1+b_0)+\rat1h(b_1-b_0)(t-h/2) 
+\rat12t(t-h)b''(\tau)
\end{equation*}
for some $0\leq\tau(t)\leq h$\,.
Then the exact change in~$X(t)$ is
\begin{align}
\Delta X=\int_0^hb(t)\,dW
={}&\rat12(b_1+b_0)\DW
+\rat1h(b_1-b_0)\int_0^h(t-h/2)\,dW 
\nonumber\\&{}
+\rat12\int_0^ht(t-h)b''(\tau)\,dW.
\label{eq:nmbdwa}
\end{align}
The error in one step of the scheme~\eqref{eq:ieuabj} is the difference between the changes~\eqref{eq:dxbdw} and~\eqref{eq:nmbdwa}.
That is, the true integral change $\Delta X=\Delta\hat X+\epsilon_0$ where the error
\begin{equation}
\epsilon_0=\frac{b_1-b_0}h\left[-\rat12S h^{3/2}+\int_0^h(t-h/2)\,dW\right]
+\rat12\int_0^ht(t-h)b''(\tau)\,dW.
\label{eq:oserrbdw}
\end{equation}
How big is this error?  First take expectations, invoke the martingale property~\eqref{eq:marty} for the two stochastic integrals, and see that $\E[\epsilon_0]=0$ provided $\E[S]=0$\,.  \emph{Thus the signs~$S$ must be chosen with mean zero.} 

Second compute the variance of the error~$\epsilon_0$ to see the size of the fluctuations in the error.  Since the expectation $\E[\epsilon_0]=0$\,, the variance $\Var[\epsilon_0]=\E[\epsilon_0^2]$.
Look at various contributions in turn.  The first term in the error~\eqref{eq:oserrbdw} has variance $\E[(Sh^{3/2})^2]=h^3\E[S^2]=\Ord{h^3}$ provided the signs~$S$ have bounded variance.  Choosing the signs~$S$ independently of the noise~$W$ there are then no correlations between the $S$~terms and the other two terms.  The second term in the error~\eqref{eq:oserrbdw} has variance
\begin{align*}
\E\left[\left(\int_0^h(t-h/2)\,dW\right)^2\right]
&{}=\int_0^h (t-h/2)^2 dt\quad\text{by It\^o isometry~\eqref{eq:isom}}
\\&{}=\rat1{12}h^3=\Ord{h^3}.
\end{align*}
The third term in the error~\eqref{eq:oserrbdw}, by the It\^o isometry~\eqref{eq:isom}, has variance
\begin{align}
\E\left[\left(\int_0^ht(t-h)b''(\tau)\,dW\right)^2\right]
&{}=\int_0^h t^2(t-h)^2b''(\tau)^2 dt
\nonumber\\&{}\leq B_2^2\int_0^h t^2(t-h)^2 dt =\frac1{30}B_2^2h^5,
\label{eq:bdwvar}
\end{align}
when the second derivative is bounded, $|b''(t)|\leq B_2$\,.
Lastly, the correlation between these previous two integrals is small as, by a slightly more general version of the It\^o isometry~\eqref{eq:isom},
\begin{align*}
&\left|\E\left[\int_0^h(t-h/2)\,dW\int_0^ht(t-h)b''(\tau)\,dW\right] \right|
\\&{}=\left|\int_0^h (t-h/2)t(t-h)b''(\tau)\,dt\right|
\\&{}\leq B_2\int_0^h \left|(t-h/2)t(t-h)\right|dt 
=\Ord{h^4}.
\end{align*}
Hence the local, one step, error is dominated by the first two contributions and has variance $\Var[\epsilon_0]=\Ord{h^3}$.

To estimate the global integral, $\intab b(t)\,dW$, we take $n=\Ord{1/h}$ time steps.  With $n$~steps the global error is the sum of $n$~local errors: the scheme~\eqref{eq:ieuabj} approximates the correct solution with global error
$\epsilon=\sum_{j=0}^{n-1}\epsilon_j$\,.  Firstly, $\E[\epsilon]=0$ as $\E[\epsilon_j]=0$ for all time steps.  Secondly, as the errors on each time step are independent, the variance
\begin{align*}
\Var[\epsilon]&{}=\sum_{j=0}^{n-1}\Var[\epsilon_j]
= n\Var[\epsilon_0]=\Ord{nh^3}=\Ord{h^2}.
\end{align*}
Thus, for the \sde\ $dX=b(t)\,dW$, the scheme~\eqref{eq:ieuabj} has global error of size~\Ord{h}.
\end{proof}

\subsection{Error for linear SDEs with additive noise}
\label{sec:elsan}

This second lemma addresses somewhat more general scalar \sde{}s.  It not only serves as a `stepping stone' to a full theorem, but illustrates two other interesting properties.  Firstly, we identify a class of \sde{}s for which the scheme~\eqref{eq:ieuabj} is second order accurate in the time step as seen in Example~\ref{sec:esoe}.  Secondly, the proof suggests that the sign~$S$ in the scheme~\eqref{eq:ieuabj} relates to sub-step properties of the noise~$W$ that are independent of the increment~$\DW$.  

\begin{lemma} \label{thm:rkaxbdw}
  The Runge--Kutta like scheme~\eqref{eq:ieuabj} has global error~\Ord{h} when applied to the additive noise, linear \sde\ $dX=a(t)X\,dt+b(t)\,dW$ for functions $a$~and~$b$ twice differentiable.   Further, in the exact differential case when $ab=db/dt$ (a solution to the \sde\  is then $X=b(t)W$) the global error is~\Ord{h^2}. 
\end{lemma}

\begin{proof}
In this case, straightforward algebra shows the first step in the scheme~\eqref{eq:ieuabj} predicts the change
\begin{align}
\Delta X={}&h\rat12(a_0+a_1)X_0
+\rat12h^2a_0a_1X_0
+\rat12(b_0+b_1)\DW
\nonumber\\&{}
+\rat12a_1b_0h(\DW-S\sqrt h)
+\rat12S\sqrt h\Delta b\,,
\label{eq:nmanls}
\end{align}
where the coefficient values $a_0=a(0)$, $a_1=a(h)$, $b_0=b(0)$ and $b_1=b(h)$.
We compare this approximate change over the time step~$h$ with the true change using iterated integrals.  For simplicity we also use subscripts to denote dependence upon `time' variables~$t$, $s$~and~$r$.  Start by writing the \sde\ $dX=a_tX_t\,dt+b_t\,dW$ as an integral over the first time step:
\begin{align}
\Delta X=
{}&\int_0^ha_tX_t\,dt +\int_0^hb_t\,dW_t
\nonumber\\&\text{[substituting $X_t=X_0+\Delta X$ inside the first integral]}
\nonumber\\={}& \int_0^ha_t\left[X_0+\int_0^ta_sX_s\,ds +\int_0^tb_s\,dW_s\right]\,dt +\int_0^hb_t\,dW_t
\nonumber\\={}& X_0\int_0^ha_t\,dt+\int_0^ha_t\int_0^ta_sX_s\,ds\,dt 
\nonumber\\&{}
+\int_0^ha_t\int_0^tb_s\,dW_s\,dt +\int_0^hb_t\,dW_t
\nonumber\\&\text{[substituting $X_s=X_0+\Delta X$ inside the second integral]}
\nonumber\\={}& X_0\int_0^ha_t\,dt
+\int_0^ha_t\int_0^ta_s\left[X_0+\int_0^sa_rX_r\,dr +\int_0^sb_r\,dW_r\right]\,ds\,dt 
\nonumber\\&{}
+\int_0^ha_t\int_0^tb_s\,dW_s\,dt 
+\int_0^hb_t\,dW_t
\nonumber\\={}& 
X_0\int_0^ha_t\,dt
+X_0\int_0^ha_t\int_0^ta_s\,ds\,dt
+\int_0^ha_t\int_0^ta_s\int_0^sa_rX_r\,dr\,ds\,dt 
\nonumber\\&{}
+\int_0^ha_t\int_0^ta_s\int_0^sb_r\,dW_r\,ds\,dt 
+\int_0^ha_t\int_0^tb_s\,dW_s\,dt 
+\int_0^hb_t\,dW_t\,.
\label{eq:nmanli}
\end{align}
For the last part of the lemma on the case of higher order error, we need to expand to this level of detail in six integrals.
Of these six integrals, some significantly match the components of the numerical step~\eqref{eq:nmanls} and some just contribute to the error.  Recall that the proof of Lemma~\ref{thm:rkbdw} identified that errors had both mean and variance. To cater for these two characteristics of errors, and with perhaps some abuse of notation, I introduce the notation~\Ord{h^p,h^q} to denote quantities with mean~\Ord{h^p} and variance~\Ord{h^q}.  For example, \Ord{h^p,0} classifies deterministic quantities~\Ord{h^p}, whereas \Ord{0,h^q}~characterises zero mean stochastic quantities of standard deviation scaling like~$h^{q/2}$.   The previous proof looked closely at the variances of error terms; here we simplify by focussing only upon their order of magnitude.  In particular, let's show that the six integrals in~\eqref{eq:nmanli} match the numerical step~\eqref{eq:nmanls} to an error~$\Ord{h^3,h^5}$.

Consider separately the integrals in~\eqref{eq:nmanli}.
\begin{itemize}
\item Firstly, $X_0\int_0^ha_t\,dt=X_0h\rat12(a_0+a_1)+\Ord{h^3}$ by the classic trapezoidal rule. This matches the first component in the numerical~\eqref{eq:nmanls}.
\item Secondly, using the linear interpolation $a_t=a_0+\frac{\Delta a}ht+\Ord{t^2}$, where as usual $\Delta a=a_1-a_0$\,, the double integral
\begin{align*}
\int_0^ha_t\int_0^ta_s\,ds\,dt
&{}=\int_0^h \Big(a_0+\frac{\Delta a}ht\Big)\Big(a_0t+\frac{\Delta a}{2h}t^2\Big)+\Ord{t^3}\,dt
\\&{}=\int_0^h a_0^2t+a_0\frac{3\Delta a}{2h}t^2 +\Ord{t^3}\,dt
\\&{}= \rat12a_0^2h^2+a_0\frac{\Delta a}{2}h^2 +\Ord{h^4}
\\&{}= \rat12a_0a_1h^2 +\Ord{h^4}.
\end{align*}
Multiplied by~$X_0$, this double integral matches the second term in the numerical~\eqref{eq:nmanls}.
\item Thirdly, the triple integral
\begin{equation*}
\int_0^ha_t\int_0^ta_s\int_0^sa_rX_r\,dr\,ds\,dt =\Ord{h^3}
\end{equation*}
because, as seen in the previous two items, each ordinary integration over a time of~$\Ord{h}$ multiplies the order of the term by a power of~$h$.

\item Fourthly, look at the single stochastic integral in~\eqref{eq:nmanli}, the last term. From the proof of the previous lemma, equations~\eqref{eq:nmbdwa} and~\eqref{eq:bdwvar} give
\begin{equation}
\int_0^hb_t\,dW_t
=\rat12(b_1+b_0)\DW
+\frac{\Delta b}h\int_0^h\big(t-\rat h2\big)\,dW_t 
+\Ord{0,h^5}.
\label{eq:nmaxbbdw}
\end{equation}
The first term here matches the third term in the numerical~\eqref{eq:nmanls}.  The second term on the right-hand side is an integral remainder that will be dealt with after the next two items.
\item Fifthly, change the order of integration in the double integral
\begin{align*}
&\int_0^ha_t\int_0^tb_s\,dW_s\,dt
\\&{}=\int_0^hb_s\int_s^ha_t\,dt\,dW_s
\\&{}=\int_0^hb_s\int_s^ha_1+\Ord{h-t}\,dt\,dW_s
\\&{}=\int_0^hb_sa_1(h-s)+\Ord{(h-s)^2}\,dW_s
\\&{}=\int_0^hb_0a_1(h-t)+\Ord{h^2}\,dW_t
\\&\qquad\text{[by the martingale property~\eqref{eq:marty} and It\^o isometry~\eqref{eq:isom}]}
\\&{}=\int_0^hb_0a_1(h-t)\,dW_t+\Ord{0,h^5}
\\&{}=\int_0^h\rat h2b_0a_1+b_0a_1\big(\rat h2-t\big)\,dW_t+\Ord{0,h^5}
\\&{}=\rat12hb_0a_1\DW-b_0a_1\int_0^h\big(t-\rat h2\big)\,dW_t+\Ord{0,h^5}
\end{align*}
The first term here matches the first part of the fourth term in the numerical~\eqref{eq:nmanls}.  The second term on the right-hand side is an integral remainder that will be dealt with after the last item.

\item Lastly, the triple integral
\begin{equation*}
\int_0^ha_t\int_0^ta_s\int_0^sb_r\,dW_r\,ds\,dt=\Ord{0,h^5}
\end{equation*}
because, as in the last item, changing the order of integration to do the stochastic integral last, the integral transforms to $\int_0^h\Ord{h^2}\,dW$ which by the martingale~\eqref{eq:marty} and It\^o isometry~\eqref{eq:isom} is~$\Ord{0,h^5}$.
\end{itemize}
Hence we now identify that the difference between the Runge--Kutta like step~\eqref{eq:nmanls} and the change~\eqref{eq:nmanli} in the true solution is the error
\begin{align}
\epsilon_0={}&{}-\rat12a_1b_0h^{3/2}S+\rat12S\sqrt h\Delta b
+b_0a_1\int_0^h\big(t-\rat h2\big)\,dW_t 
\nonumber\\&{}
-\frac{\Delta b}h\int_0^h\big(t-\rat h2\big)\,dW_t
+\Ord{h^3}+\Ord{0,h^5}
\nonumber\\={}&
\left[\rat12Sh^{3/2}-\int_0^h\big(t-\rat h2\big)\,dW_t\right]
\left\{-a_1b_0+\frac{\Delta b}h\right\}
+\Ord{h^3,h^{5}}.
\label{eq:nmaxbdwd}
\end{align}
Two cases arise corresponding to the main and the provisional parts of lemma~\ref{thm:rkaxbdw}.
\begin{itemize}
\item In the general case, the factor in square brackets,~$[\cdot]$, in~\eqref{eq:nmaxbdwd} determines the order of error.  Choosing the signs~$S$ randomly with mean zero then $Sh^{3/2}=\Ord{0,h^3}$. Recall the integral $\int_0^h\big(t-\rat h2\big)\,dW_t=\Ord{0,h^3}$ also. Thus the leading error is then~\Ord{h^3,h^3}.  This is the local one step error.  Summing over \Ord{1/h} time steps gives that the global error is~\Ord{h^2,h^2}.  That is, the error due to the noise dominates, variance~\Ord{h^2}, and is generally first order in~$h$ as the standard deviation of the error is of order~$h$.

But as the noise decreases to zero, $b\to0$, the factor in curly braces,~$\{\cdot\}$, goes to zero.  In this decrease the order of error~\eqref{eq:nmaxbdwd} transitions smoothly to the deterministic case of local error~\Ord{h^3} and hence global error~\Ord{h^2}.

\item The second case is when the factor in braces in~\eqref{eq:nmaxbdwd} is small: this occurs for the integrable case $ab=db/dt$ as then the term in braces is~\Ord{h} so that the whole error~\eqref{eq:nmaxbdwd} becomes~\Ord{h^3,h^5}.  Again this is the local one step error.  Summing over \Ord{1/h}~time steps gives that the global error is~\Ord{h^2,h^4}.  That is, in this case the error is of second order in time step~$h$, both through the deterministic error and the variance of the stochastic errors.  Figure~\ref{fig:sde1abll1} shows another case when the error is second order.
\end{itemize}
This concludes the proof.
\end{proof}

Interestingly, we would decrease the size of the factor in brackets in the error~\eqref{eq:nmaxbdwd} by choosing the sign~$S$ to cancel as much as possible the integral $\int_0^h\big(t-\rat h2\big)\,dW_t$\,.  This sub-step integral is one characteristic of the sub-step structure of the noise, and is independent of~$\DW$.  If we knew this integral, then we could choose the sign~$S$ to cause some error cancellation; however, generally we do not know the sub-step integral.  But this connection between the signs~$S$ and the integral $\int_0^h\big(t-\rat h2\big)\,dW_t$ does suggest that the sign~$S$ relates to sub-step characteristics of the noise process~$W$.

For example, if one used \idx{Brownian bridge}s to successively refine the numerical approximations for smaller and smaller time steps, then it may be preferable to construct a Brownian bridge compatible with the signs~$S$ used on the immediately coarser step size.\footnote{The Brownian bridge stochastically interpolates a Wiener process to half-steps in time if all one knows is the increment~$\DW$ over a time step~$h$.  The Brownian bridge asserts that the change over half the time step,~$h/2$, is $\rat12\DW-\rat12\sqrt hZ$ for some $Z\sim N(0,1)$; the change over the second half of the time step is correspondingly $\rat12\DW+\rat12\sqrt hZ$\,.  Factoring out the half, these sub-steps are $\rat12(\DW\mp Z\sqrt h)$ which match the factors $(\DW \mp S\sqrt h)$ used by the scheme~\eqref{eq:ieuabj}: the discrete signs~$S=\mp1$ have mean zero and variance one just like the normally distributed~$Z$ of the Brownian bridge.}

\subsection{Global error for general SDEs}
\label{sec:gegs}

The previous section~\ref{sec:elsan} established the order of error for a special class of linear \sde{}s.  The procedure is to repeatedly substitute integral expressions for the unknown whereever it appears (analogous to Picard iteration).  In section~\ref{sec:elsan} each substitution increased the number of integrals in the expression by two.  For general \sde{}s, this subsection employs the same procedure, but now the number of integrals doubles in each substitution.  The rapid increase in the number of integrals is a major complication, so we only consider the integrals necessary to establish that the global error is~\Ord{h}.  

Further, the following theorem is also proven for vector \sde{}s in~$\mathbb R^n$, whereas the previous two subsection sections only considered special scalar \sde{}s.

\begin{theorem} \label{thm:nm1ito}
The Runge--Kutta like numerical scheme~\eqref{eq:ieuabj} generally has global error~\Ord{h} when applied to the \sde~\eqref{eq:sde1ab} for sufficiently smooth drift and volatility functions $\vec a(t,x)$~and~$\vec b(t,x)$. 
\end{theorem}

\begin{proof}
The proof has two parts:  the first is the well known, standard, expansion of the solution of the general \sde~\eqref{eq:sde1ab} by iterated stochastic integrals leading to the Milstein scheme~\cite[e.g.]{Higham01, Kloeden01}; the second shows how the scheme~\eqref{eq:ieuabj} matches the integrals to an order of error.

First look at the repeated integrals for one time step; without loss of generality, start with a time step from $t_0=0$ to $t_1=t_0+h=h$ as the analysis for all other time steps is identical with minor shifts in the times of evaluation and integration. 
The stochastic `Taylor series' analysis starts from the integral form of It\^o formula~\eqref{eq:ito}: for a stochastic process~$\vec X(t)$ satisfying the general It\^o \sde~\eqref{eq:sde1ab},  for operators~$L_t^k$, any smooth function~$f(t,\vec x)$ of the process satisfies
\begin{align}&
f(t,\vec X_t)=f(0,\vec X_0)+\int_0^t L_s^0f(s,\vec X_s)\,ds+\int_0^tL^1_sf(s,\vec X_s)\,dW_s\,,
\label{eq:itol}
\\\text{where}\quad&
L^0_s=\left[\D t{}+a_i\D {x_i}{}+\frac12b_ib_j\Dx{x_i}{x_j}{}\right]_{t=s},
\quad
L^1_s=\left[b_i\D {x_i}{}\right]_{t=s}.
\nonumber
\end{align}
For conciseness we use subscripts~$0$, $t$, $s$ and~$r$ to denote evaluation at these times, and similarly $f_t=f(t,\vec X_t)$, and use subscripts~$i$ and~$j$ to denote components of a vector, with the Einstein summation convention for repeated indices.  
As you would expect, when stochastic effects are absent, $\vec b=\vec 0$\,, the integral formula~\eqref{eq:itol} reduces, through the first two components of~$L^0_s$, to an integral version of the well known deterministic chain rule: $f(t,\vec X_t)=f(0,\vec X_0)+\int_0^t \big[\partial_tf(s,\vec X_s) +a_i\partial_{x_i}f(s,\vec X_s)\big]\,ds$\,.
Now turn to the \sde~\eqref{eq:sde1ab} itself: it is a differential version of an integral equation which over the first time step gives
\begin{align}
\Delta \vec X={}& \vec X(h,\omega)-\vec X(0,\omega)
=\int_0^h d\vec X
\nonumber\\={}&
\int_0^h\vec a(t,\vec X_t)\,dt+\int_0^h\vec b(t,\vec X_t)\,dW_t
\nonumber\\&[\text{apply the It\^o formula~\eqref{eq:itol} to both }\vec a(t,\vec X_t)\text{ and }\vec b(t,\vec X_t)]
\nonumber\\={}&
\int_0^h \left[ \vec a_0+\int_0^t L_s^0\vec a_s\,ds+\int_0^tL^1_s\vec a_s\,dW_s \right]\,dt
\nonumber\\&{}
+\int_0^h \left[ \vec b_0+\int_0^t L_s^0\vec b_s\,ds+\int_0^tL^1_s\vec b_s\,dW_s\right]\,dW_t
\nonumber\\&[\text{apply the It\^o formula~\eqref{eq:itol} to }L^1_s\vec b_s]
\nonumber\\={}&
\int_0^h \vec a_0\,dt
+\int_0^h\int_0^t L_s^0\vec a_s\,ds\,dt
+\int_0^h\int_0^tL^1_s\vec a_s\,dW_s\,dt
\nonumber\\&{}
+\int_0^h \vec b_0\,dW_t
+\int_0^h\int_0^t L_s^0\vec b_s\,ds\,dW_t
\nonumber\\&{}
+\int_0^h\int_0^t\left[ L^1_0\vec b_0 +\int_0^s L_r^0L_r^1\vec b_r\,dr+\int_0^sL^1_rL_r^1\vec b_r\,dW_r\right]\,dW_s\,dW_t
\nonumber\\&[\text{now rearrange these eight integrals in order of magnitude}]
\nonumber\\={}&
\vec a_0\int_0^h dt
+\vec b_0\int_0^h dW_t
+L^1_0\vec b_0\int_0^h\int_0^tdW_s\,dW_t
\nonumber\\&{}
+\left[
\int_0^h\int_0^tL^1_s\vec a_s\,dW_s\,dt
+\int_0^h\int_0^t L_s^0\vec b_s\,ds\,dW_t
\right.\nonumber\\&\left.\qquad{}
+\int_0^h\int_0^t\int_0^sL^1_rL_r^1\vec b_r\,dW_r\,dW_s\,dW_t
\right]
\nonumber\\&{}
+\left\{
\int_0^h\int_0^t L_s^0\vec a_s\,ds\,dt
+\int_0^h\int_0^t\int_0^s L_r^0L_r^1\vec b_r\,dr\,dW_s\,dW_t
\right\}
\label{eq:dvie}
\end{align}
\begin{itemize}
\item Simplify the first line in this last expression~\eqref{eq:dvie} for~$\Delta\vec X$ using the well known integrals $\int_0^h dt=h$\,, $\int_0^hdW_t=\DW$ and $\int_0^h\int_0^tdW_s\,dW_t=\int_0^hW_t\,dW_t=\rat12(\DW^2-h)$ \cite[(3.6), e.g.]{Higham01}.  The last of these three integrals follow from applying It\^o's formula applied to $F(t,W_t)=\rat12W_t^2$ to deduce $dF=\rat12\,dt+W_t\,dW_t$\,, and integrating a rearrangement gives $\int W_t\,dW_t=\int dF-\int\rat12\,dt=\rat12W_t^2-\rat12t$\,.  Also simplify the first line by defining the matrix $\vec b'_0=\begin{bmatrix} \D{x_j}{b_i} \end{bmatrix}_{t=0}$ so that $L^1_0\vec b_0=\vec b'_0\vec b_0$\,.
\item The three integrals above in square brackets in expression~\eqref{eq:dvie} all have expectation zero and variance~\Ord{h^3}.  Recall that with two arguments \Ord{h^p,h^q}~denotes quantities with mean~\Ord{h^p} and variance~\Ord{h^q}. Thus these three integrals in square brackets are~\Ord{0,h^3}.    
\item The two integrals above in curly braces in expression~\eqref{eq:dvie} are all~\Ord{h^2} in magnitude and hence are~\Ord{h^2,h^4}.  
\end{itemize}
Combining all these leads to the well established Milstein scheme for the change in~$\vec X$ over one time step from~$t_0$ to~$t_1$ as
\begin{equation}
\Delta \vec X=\vec a_0h
+\vec b_0\DW
+\vec b'_0\vec b_0\rat12(\DW^2-h)
+\Ord{h^2,h^3}.
\label{eq:mils}
\end{equation}

Second, we proceed to show the scheme~\eqref{eq:ieuabj} matches this Milstien scheme~\eqref{eq:mils}. 
Note $\vec K_1=h\vec a_0+(\DW-S\sqrt h)\vec b_0=\Ord{h,h}$ so the product $\vec K_1\vec K_1=\Ord{h,h^2}$ and so on.  Hence, by Taylor series in the arguments of the smooth drift~$\vec a$ and volatility~$\vec b$, 
\begin{align*}
\vec K_2={}& h\left[\vec a_0+\vec a'_0\vec K_1+\Ord{h,h^2}\right]
\\&{}+(\DW+S\sqrt h)\left[\vec b_0+h\dot{\vec b}_0+\vec b'_0\vec K_1
+\rat12\vec b''_0\vec K_1\vec K_1 +\Ord{h^2,h^3} \right]
\end{align*}
where $\vec b''\vec K_1\vec K_1$ denotes the tensorial double sum~$\Dx{x_i}{x_j}{\vec b}K_{1i}K_{1j}$, and where the overdot denotes the partial derivative with respect to time,~$\dot{\vec b}=\D t{\vec b}$.
Combining $\vec K_1$~and~$\vec K_2$, the corresponding first step in the scheme~\eqref{eq:ieuabj} predicts the change
\begin{align}
\Delta \vec X={}&\vec a_0h+ \vec b_0\DW +\rat12\vec b'_0\vec b_0(\DW^2-S^2h)
\nonumber\\&{}
+\rat12(\DW-S\sqrt h)\left[ h\vec a'_0\vec b_0+\rat12(\DW^2-S^2h)\vec b''_0\vec b_0\vec b_0\right]
\nonumber\\&{}
+\rat12h(\DW+S\sqrt h)(\dot {\vec b}_0+\vec b'_0\vec a_0)
+\Ord{h^2,h^4}\,.
\label{eq:nmrkab}
\end{align}
Provided $S^2=1+\Ord{h,h}$ the first lines match to~\Ord{h^2,h^3}: normally $S^2=1$ as specified in~\eqref{eq:ieuabj}. Other terms detailed in~\eqref{eq:nmrkab} are~\Ord{0,h^3} provided $\operatorname{E}(S)=\Ord{0,h}$: normally set to be zero as specified in~\eqref{eq:ieuabj}.  Hence one step of the scheme~\eqref{eq:ieuabj} matches the solution to~\Ord{h^2,h^3}. The local error over one step of~\Ord{h^2,h^3} leads to, over \Ord{1/h}~steps, a global error of \Ord{h,h^2}.
\end{proof}

This proof confirms the order of error seen in the earlier examples.  Further, because we can readily transform between It\^o and Stratonovich \sde{}s, we now prove that a minor variation of the numerical scheme applies to Stratonovich \sde{}s.

\begin{corollary}[Stratonovich SDEs] \label{cor:nm1strat}
   The Runge--Kutta like scheme~\eqref{eq:ieuabj}, but setting $S=0$\,, has errors~\Ord{h} when the \sde~\eqref{eq:sde1ab} is to be interpreted in the Stratonovich sense.
\end{corollary}

\begin{proof}
Interpreting the \sde~\eqref{eq:sde1ab} in the Stratonovich sense implies solutions are the same as the solutions of the It\^o \sde
\begin{equation*}
d\vec X=(\vec a+\rat12\vec b'\vec b)\,dt+\vec b\,dW.
\end{equation*}
Apply the scheme~\eqref{eq:ieuabj} (with $S=\pm1$ as appropriate to an It\^o \sde), or the analysis of the previous proof, to this It\^o \sde.  Then, for example, the one step change~\eqref{eq:nmrkab} becomes
\begin{equation*}
\Delta \vec X=(\vec a_0+\rat12\vec b'_0\vec b_0)h+\vec b_0\DW+\rat12\vec b'_0\vec b_0(\DW^2-h)+\Ord{h^2,h^3}.
\end{equation*}
The component of the deterministic drift term that involves~$\vec b_0\vec b'_0$ cancel leaving, in terms of the coefficient functions of the Stratonovich \sde,
\begin{equation}
\Delta \vec X=a_0h+\vec b_0\DW+\rat12\vec b'_0\vec b_0\DW^2 +\Ord{h^2,h^3}.
\label{eq:nm1strat}
\end{equation}
Now apply the scheme~\eqref{eq:ieuabj} with $S=0$ to the Stratonovich \sde:
Taylor series expansions obtain the one step numerical prediction as~\eqref{eq:nmrkab} upon setting $S=0$\,.  This one step numerical prediction is the same as~\eqref{eq:nm1strat} to the same order of errors.  Thus the  scheme~\eqref{eq:ieuabj} with $S=0$ solves the Stratonovich interpretation of the \sde~\eqref{eq:sde1ab}.
\end{proof}

\begin{exercise}[iterated integrals] 
Consider the scalar \sde\ $dX=X\,dW$.  This \sde\ is shorthand for the It\^o integral $X_t=X_0+\intw s{X_s}$\,.
Over a small time interval~$\Delta t=h$ this integral gives $X_h=X_0+\intw t{X_t}$\,.  Use this as the start of an iteration to provide successively more accurate approximations to~$X_h$: successive approximations are successive truncations of
\begin{equation*}
X_h\approx X_0+X_0\, \DW
+X_0\left[\rat12(\DW )^2-\rat12h\right]
+X_0\left[\rat16(\DW )^3-\rat12h\DW \right].
\end{equation*}
Determine the integral remainders for each of the approximations. 
\end{exercise}

\begin{exercise}[quadratic convergence]
Adapt the proof of Lemma~\ref{thm:rkaxbdw} to prove that in the specific case when the drift~$\vec a=\vec\alpha(t)+\beta(t)\vec X$  and the volatility, independent of~$x$, satisfies $\dot {\vec b}=\beta \vec b$, then the scheme has local error~$\Ord{h^3,h^5}$ and hence global error~$\Ord{h^2,h^4}$, as seen in Figure~\ref{fig:sde1abll1}. 

\end{exercise}

\section{Conclusion}

A good basic numerical scheme for integrating It\^o \sde{}s is the Runge--Kutta like scheme~\eqref{eq:ieuabj} (set $S_k=0$ to integrate Stratonovich \sde{}s).  A teacher could introduce it in the context of the introduction to numerical \sde{}s outlined by Higham~\cite{Higham01}. 

One of the appealing features of the scheme~\eqref{eq:ieuabj} is that it reduces,  for small noise, to a well known scheme for deterministic \ode{}s. Consequently, we expect the global error \Ord{\|\vec a\|h^2+\|\vec b\|h} for some norms of the drift and volatility.  Such more general expressions of the error should be useful in multiscale simulations where the strength of the noise depends upon the macroscale time step, such as in the modelling of a stochastic Hopf bifurcation~\cite[\S5.4.2]{Roberts06k}.

One required extension of the scheme~\eqref{eq:ieuabj} is to generalise it, if possible, to the case of multiple independent noises.  I am not aware of an attractive generalisation to this practically important case.

\bibliographystyle{plain}
\bibliography{ajr,bib}

\end{document}